\documentclass[12pt]{amsart}
\usepackage[all]{xy}
\usepackage[parfill]{parskip}

\usepackage{color}

\usepackage{amsmath, amscd, graphicx, latexsym, hyperref, times}
\usepackage{color,graphicx}
\usepackage[abs]{overpic}
\usepackage{tikz}
\usepackage{tikz-cd}
\usepackage{epstopdf}
\usetikzlibrary{arrows, patterns}

\oddsidemargin .25in \theoremstyle{plain}
\newtheorem{dummy}{anything}[section]
\newtheorem{theorem}[dummy]{Theorem}

\newtheorem{lemma}[dummy]{Lemma}

\newtheorem{proposition}[dummy]{Proposition}

\newtheorem{corollary}[dummy]{Corollary}

\theoremstyle{definition}

\newtheorem{example}[dummy]{Example}

\newtheorem{remark}[dummy]{Remark}

\newtheorem*{acknowledgements}{Acknowledgements}

\newcommand{\Z}{\mathbb{Z}}

\newcommand{\F}{\mathbb{F}}

\newcommand{\HF}{\widehat{HF}}

\newcommand{\spin}{\mathrm{Spin}^c}

\begin{document}

\title{Contact (+1)-surgeries along Legendrian Two-component Links}

\author{Fan Ding, Youlin Li and Zhongtao Wu}

\address{School of Mathematical Sciences and LMAM, Peking University, Beijing 100871, China}
\email{dingfan@math.pku.edu.cn}

\address{School of Mathematical Sciences, Shanghai Jiao Tong University, Shanghai 200240, China}
\email{liyoulin@sjtu.edu.cn}

\address{Department of Mathematics, The Chinese University of Hong Kong, Shatin, Hong Kong}
\email{ztwu@math.cuhk.edu.hk}


\begin{abstract}
In this paper, we study contact surgeries along Legendrian links in the standard contact 3-sphere. On one hand, we use algebraic methods to prove the vanishing of the contact Ozsv\'{a}th-Szab\'{o} invariant for contact $(+1)$-surgery along certain Legendrian two-component links. The main tool is a link surgery formula for Heegaard Floer homology developed by Manolescu and Ozsv\'{a}th. On the other hand, we use contact-geometric argument to show the overtwistedness of the contact 3-manifolds obtained by contact $(+1)$-surgeries along Legendrian two-component links whose two components are linked in some special configurations.


\end{abstract}


\maketitle

\section{Introduction}\label{sec: intro}

A \emph{contact structure} $\xi$ on a smooth oriented 3-manifold $Y$ is a smooth tangent 2-plane field $\xi$ such that any smooth 1-form $\alpha$ locally defining $\xi$ as $\xi=\ker\alpha$ satisfies the condition $\alpha \wedge d\alpha >0$. A contact structure $\xi$ is coorientable if and only if there is a global 1-form $\alpha$ with $\xi=\ker\alpha$. Throughout this paper, we will assume our 3-manifolds are oriented, connected and our contact structures are cooriented.
A contact structure $\xi$ on $Y$ is called \emph{overtwisted} if one can find an embedded disc $D$ in $Y$ such that the tangent plane field of $D$ along its boundary coincides with $\xi$; otherwise, it is called \emph{tight}.  Any closed oriented 3-manifold admits an overtwisted contact structure (cf. \cite{ge}). It is much harder to find tight contact structures on a closed oriented 3-manifold. The following question is still open: Which closed oriented 3-manifolds admit tight contact structures?

One way of obtaining new contact manifolds from the existing one is through \emph{contact surgery}.  Suppose $L$ is a \emph{Legendrian knot} in a contact 3-manifold $(Y,\xi)$, i.e., $L$ is tangent to the given contact structure $\xi$ on $Y$. Contact surgery is a version of Dehn surgery that is adapted to the contact category.  Roughly speaking, we delete a tubular neighborhood of $L$, then reglue it, and obtain a contact structure on the surgered manifold by extending $\xi$ from the complement of the tubular neighborhood of $L$ to a tight contact structure on the reglued solid torus (see \cite{dg} for details). In \cite{dg}, the first author and Geiges proved that every closed contact 3-manifold $(Y,\xi)$ can be obtained by contact $(\pm 1)$-surgery along
a Legendrian link in $(S^3, \xi_{std})$, where $\xi_{std}$ denotes the standard contact structure on $S^3$.


Many tools have been developed to detect tightness, including an invariant $c(Y, \xi)\in \HF(-Y)$ in Heegaard Floer theory for closed contact 3-manifolds $(Y,\xi)$. We call it the \emph{contact Ozsv\'ath-Szab\'o invariant}, or simply the \emph{contact invariant} of $(Y,\xi)$.  It is shown that $c(Y, \xi)=0$ if $(Y,\xi)$ is overtwisted \cite{OSzContact}, and $c(Y, \xi)\neq 0$ if $(Y,\xi)$ is strongly symplectically fillable \cite{gh}.

It is natural to ask whether the contact invariant of a contact 3-manifold obtained by contact surgery along a Legendrian link is trivial or not.  All known results concern contact surgeries along Legendrian knots. In \cite{ls}, Lisca and Stipsicz showed that contact $\frac{1}{n}$-surgeries along certain Legendrian knots in $(S^{3}, \xi_{std})$ yield contact 3-manifolds with nonvanishing contact invariants for any positive integer $n$. In \cite{g}, Golla considered contact 3-manifolds obtained from $(S^3,\xi_{std})$ by contact $n$-surgeries along  Legendrian knots, where $n$ is any positive integer. He gave a necessary and sufficient condition for the contact invariant of such a contact 3-manifold to be nonvanishing. In \cite{mt}, Mark and Tosun extended Golla's result to contact $r$-surgeries, where $r>0$ is rational.

To go further along this line of investigation, we study contact $(+1)$-surgeries along Legendrian two-component links in $(S^{3}, \xi_{std})$ in this paper. Here, contact $(+1)$-surgery along a Legendrian link means contact $(+1)$-surgery along each component of the Legendrian link.  One of our main results below (Theorem \ref{thm: main}) can be viewed as a first step towards our ultimate goal of obtaining a necessary and sufficient condition for contact $(+1)$-surgery on a link to yield a manifold with nonvanishing contact Ozsv\'ath-Szab\'o invariant.


\begin{theorem} \label{thm: main} Suppose $L=L_1\cup L_2$ is a Legendrian two-component link in the standard contact 3-sphere $(S^3, \xi_{std})$ whose two components have nonzero linking number.
Assume $L_2$ satisfies $\nu^+(L_2)=\nu^+(\overline{L_2})=0$, where $\overline{L_2}$ denotes the mirror of $L_2$. Then contact $(+1)$-surgery on $(S^3, \xi_{std})$ along $L$ yields a contact 3-manifold with vanishing contact invariant.
\end{theorem}

The main tool for proving this theorem is a link surgery formula for the Heegaard Floer homology of integral surgeries on links developed by Manolescu and Ozsv\'ath \cite{mo}.
Here, $\nu^+$ is a numerical invariant defined by Hom and the third author in \cite{HW} based on work of Rasmussen \cite{Ras}.  It is shown in \cite[Proposition 3.11]{h} that a knot $K$ satisfies the condition $\nu^+(K)=\nu^+(\overline{K})=0$ if and only if we have a filtered chain homotopy equivalence
\begin{equation} \label{equivalence}
CFK^\infty(K) \simeq CFK^\infty (U) \oplus A
\end{equation}
where $U$ denotes the unknot and $A$ is acyclic, i.e., $H_\ast(A)=0$.  In particular, such a knot must have $\tau(K)=0$.  Applying (\ref{equivalence}) enables us to treat $K$ effectively like the unknot in the proof of Theorem~\ref{thm: main}. 

Below, we list some interesting families of knots that satisfy the condition $$\nu^+(K)=\nu^+(\overline{K})=0$$ of Theorem \ref{thm: main}.

\begin{example} \label{Example1}
\begin{enumerate}

\item The most basic examples are the \emph{slice knots}.

We are particularly interested in Legendrian slice knots with Thurston-Bennequin invariant $-1$, as contact $(+1)$-surgeries along these knots result in contact 3-manifolds with nonvanishing contact invariants \cite{g}. Nontrivial knot types of smoothly slice knots with at most 10 crossings that have Legendrian representatives with Thurston-Bennequin invariant $-1$ are $\overline{9_{46}}$ (the mirror of $9_{46}$) and $\overline{10_{140}}$ \cite{cns}. In Figures~\ref{fig: l1} and \ref{fig: l2} below, we show a couple of Legendrian two-component links in $(S^3,\xi_{std})$ that include a Lengendrian unknot and a Legendrian knot of type $\overline{9_{46}}$, respectively.  Note that one obtains nonvanishing contact invariant after performing contact $(+1)$-surgery along each knot component of the depicted links.  On the other hand, Theorem~\ref{thm: main} implies that contact $(+1)$-surgeries along these links result in contact 3-manifolds with vanishing contact invariants.

\begin{figure}[htb]
\begin{overpic}
{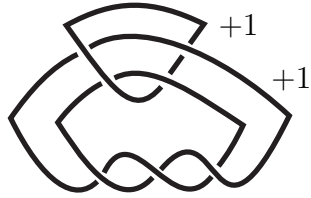}
\put(80, 60){$+1$}
\put(100, 40){$+1$}

\end{overpic}
\caption{The upper component is a Legendrian unknot with Thurston-Bennequin invariant $-1$. }
\label{fig: l1}
\end{figure}

\begin{figure}[htb]
\begin{overpic}
{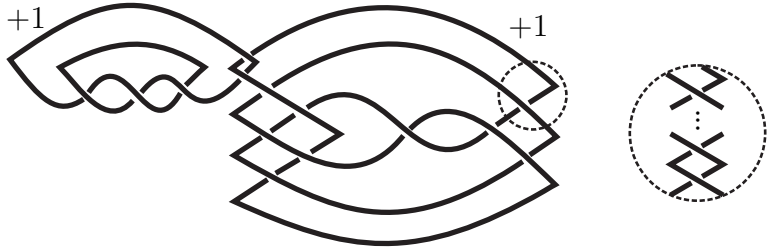}
\put(0, 83){$+1$}
\put(190, 80){$+1$}
\put(260, 42){$\cdot$}
\put(260, 48){$\cdot$}
\put(260, 45){$\cdot$}

\end{overpic}
\caption{The right component of the link is a Legendrian knot of type $\overline{9_{46}}$ with Thurston-Bennequin invariant $-1$ and rotation number $0$. The right figure is a Legendrian tangle. }
\label{fig: l2}
\end{figure}

As pointed out by the referee, if we replace the dashed circled area of the front projection diagram of the Legendrian $\overline{9_{46}}$ by the right tangles in Figure~\ref{fig: l2}, then we can obtain infinitely many prime Legendrian slice knots with Thurston-Bennequin invariant $-1$.  Alternatively, as $tb(L_{1}\sharp L_{2})=tb(L_{1})+tb(L_{2})+1$, where $L_{1}\sharp L_{2}$ denotes the Legendrian connected sum \cite{eh}, we can obtain infinitely many composite Legendrian slice knots with Thurston-Bennequin invariant $-1$.  In either case, we can create infinitely  families of examples similar to Figures~\ref{fig: l1} and \ref{fig: l2} using those Legendrian slice knots.


\item More generally, all \emph{rationally slice knots} satisfy $\nu^+(K)=\nu^+(\overline{K})=0$ \cite[Theorem 1.4]{KW}.

Recall that a knot $K\subset S^3$ is rationally slice if there exists an embedded disk $D$ in a rational homology $4$-ball $V$ such that $\partial (V,D)=(S^3, K)$.  Examples of rationally slice knots include \emph{strongly $-$ amphicheiral knots} and \emph{Miyazaki knots}, i.e., fibered, $-$ amphicheiral knots with irreducible Alexander polynomial \cite{KW}.  In particular, the figure-eight knot is rationally slice but not slice.

 \end{enumerate}

\end{example}

\begin{example}\label{Example2}
Let $L=L_{1}\cup L_{2}$ be a Legendrian link in $(S^{3}, \xi_{std})$. Suppose $L_{2}$ is a Legendrian unknot with $tb(L_{2})=-1$, and $L_2$ is a meridional curve of $L_1$. Then, by Theorem~\ref{thm: main}, contact $(+1)$-surgery along $L$ yields a contact structure $\xi_L$ on $S^3$ with vanishing contact invariant. Hence by the classification of tight contact structures on $S^3$ \cite{e0} and $c(S^3,\xi_{std})\neq 0$, the contact 3-manifold $(S^3,\xi_L)$ is overtwisted.
\end{example}

In the special case of Theorem \ref{thm: main} where $L_2$ is a Legendrian unknot with $tb(L_2)=-1$,  contact $(+1)$-surgery on $(S^3, \xi_{std})$ along $L_2$ yields the unique (up to isotopy) tight contact structure $\xi_{t}$ on $S^{1}\times S^{2}$.  Hence, in this case, we may interpret the theorem as a result of contact $(+1)$-surgery along a Legendrian knot in $(S^{1}\times S^{2}, \xi_{t})$.  More generally, we have the following corollary.

\begin{corollary} \label{cor} Suppose $L$ is a Legendrian knot in $\sharp^{k}(S^{1}\times S^{2}, \xi_{t})$, the contact connected sum of $k$ copies of  $(S^{1}\times S^{2},\xi_{t})$. If $L$ is not null-homologous, then contact $\frac{1}{n}$-surgery on $\sharp^{k}(S^{1}\times S^{2}, \xi_{t})$ along $L$ yields a contact 3-manifold with vanishing contact invariant for any positive integer $n$.
\end{corollary}

We also consider contact $(+1)$-surgeries along Legendrian two-component links in $(S^3,\xi_{std})$ whose two components have zero linking number. Unlike the nonzero linking number case, there are numerous instances for which the contact invariant is nonvanishing even if one of the components of the surgered link is the unknot.  For example, $(+1)$-surgery along the Legendrian two-component unlink with Thurston-Bennequin invariant $-1$ for both components yields $\sharp^{2}(S^{1}\times S^{2}, \xi_{t})$. Another more interesting example comes from contact $(+1)$-surgery along the Legendrian two-component link as depicted in Figure~\ref{fig: lk4}.  The resulting contact 3-manifold has nonvanishing contact invariant by \cite[Exercise 12.2.8(c)]{ozst}. 






\begin{figure}[htb]
\begin{overpic}
{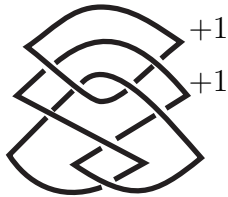}
\put(70, 60){$+1$}
\put(70, 40){$+1$}

\end{overpic}
\caption{The lower component is a Legendrian right handed trefoil with Thurston-Bennequin invariant $1$ and rotation number $0$.  }
\label{fig: lk4}
\end{figure}

Thus, we expect that analogous results to Theorem \ref{thm: main} for the linking number 0 case are likely more difficult to obtain.  Instead, we will first study some special cases in our paper, namely contact $(+1)$-surgeries along Legendrian Whitehead links.  To the best of our knowledge, the contact invariants or tightness of such manifolds have not been explicitly given in the literature.

\begin{proposition}\label{PropWhitehead}
Contact $(+1)$-surgery on $(S^3, \xi_{std})$ along a Legendrian Whitehead link yields a contact 3-manifold with vanishing contact invariant.
\end{proposition}

In light of the above vanishing results in Theorem \ref{thm: main} and Proposition \ref{PropWhitehead}, one may wonder whether the contact manifolds studied there are tight or not. As far as we know, this question is wide open even for contact $(+1)$-surgeries along Legendrian knots. Two sufficient conditions for contact $(+1)$-surgeries along Legendrian knots to be overtwisted were given by \"{O}zba\u{g}ci in \cite{oz} and by Lisca and Stipsicz in \cite[Theorem 1.1]{ls3}.   We come up with the following sufficient condition for contact $(+1)$-surgeries along Legendrian two-component links yielding overtwisted contact 3-manifolds.  Note that this theorem is irrelevant to the linking number of the two components of the Legendrian link along which we perform contact $(+1)$-surgery. It is inspired by the work of Baker and Onaran \cite[Proposition 4.1.10]{bo}.

\begin{theorem}\label{overtwisted}
Suppose there exists a front projection of  a Legendrian two-component link $L=L_1\cup L_2$ in the standard contact 3-sphere $(S^3, \xi_{std})$ that contains one of the configurations exhibited in Figure~\ref{fig: ot3}, then contact $(+1)$-surgery on $(S^3, \xi_{std})$ along $L$ yields an overtwisted contact 3-manifold.
\end{theorem}

\begin{figure}[htb]
\begin{overpic}
{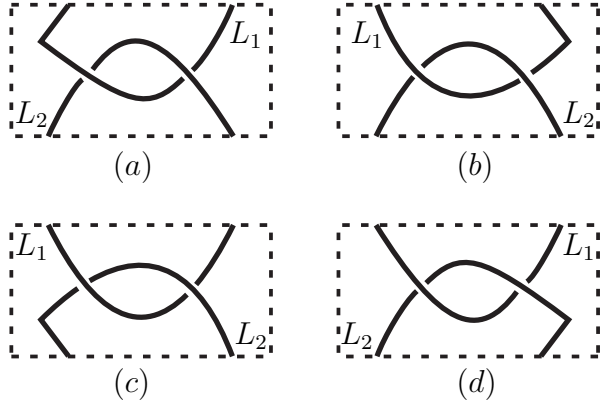}

\put(84, 120){$L_{1}$}
\put(3, 90){$L_{2}$}
\put(40, 70){$(a)$}

\put(130, 120){$L_{1}$}
\put(210, 90){$L_{2}$}
\put(170, 70){$(b)$}


\put(3, 40){$L_{1}$}
\put(86, 6){$L_{2}$}
\put(40, -12){$(c)$}

\put(210, 40){$L_{1}$}
\put(126, 6){$L_{2}$}
\put(170, -12){$(d)$}

\end{overpic}
\vspace{3mm}
\caption{Four configurations in a front projection of a Legendrian two-component link $L$. }
\label{fig: ot3}
\end{figure}

As an application, we will show in Examples~\ref{Example4} and \ref{Example5} that contact $(+1)$-surgeries along the Legendrian links in Figures~\ref{fig: l1} and \ref{fig: l2} actually yield overtwisted contact 3-manifolds.

The remainder of this paper is organized as follows.  In Section 2, we review basic properties of the contact invariant.  We also reformulate the statement of Golla concerning the conditions under which contact $(+1)$-surgery along a Legendrian knot yields a contact 3-manifold with nonvanishing contact invariant. In Section 3, we go through the construction of the link surgery formula of Manolescu and Ozsv\'ath in the special case of two-components links. We elaborate on the $E_1$ page of an associated spectral sequence and identify the relevant maps in the differential $\partial_1$ with the well-known $\hat{v}$ and $\hat{h}$ in the knot surgery formula of Ozsv\'ath and Szab\'o \cite{OSzSurgery}. In Section 4, we analyze the $E_1$ page and give a proof of Theorem \ref{thm: main} based on diagram chasing.  This idea is partly inspired by the work of Mark-Tosun \cite{mt} and Hom-Lidman \cite{HL}, and also constitutes the most novel part of our paper.
Due to some technical issues, the above argument does not apply to the linking number 0 case, so in Section 5, we use a different machinery in Heegaard Floer homology, namely the grading, to prove Proposition \ref{PropWhitehead}.  Finally, we prove Theorem~\ref{overtwisted} in Section 6.  Applying Legendrian Reidemeister moves, we obtain more examples of overtwisted contact $(+1)$-surgery in Corollary~\ref{overtwisted1}.

\medskip
\begin{acknowledgements} The authors would like to thank Ciprian Manolescu, Yajing Liu, Tye Lidman, Jen Hom, Faramarz Vafaee,  Eugene Gorsky and John Etnyre for helpful discussions and suggestions.
Part of this work was done while the second author was visiting The Chinese University of Hong Kong, and he would like to thank for its generous hospitality and support.  The first author was partially supported by Grant No.\ 11371033 of the National Natural Science Foundation of China. The second author was partially supported by Grant No.\ 11471212 and No.\ 11871332 of the National Natural Science Foundation of China. The third author was partially supported by grants from the Research Grants Council of the Hong Kong Special Administrative Region, China (Project No.\ 14300018). Last but not least, we would like to thank the referees for valuable comments and suggestions.
\end{acknowledgements}

\section{Preliminaries on contact invariants}

We briefly review backgrounds in Heegaard Floer homology and the contact invariant in this section.  Throughout this paper, we work with Heegaard Floer homology with coefficients in $\mathbb{F}=\mathbb{Z}/2\mathbb{Z}$.  Heegaard Floer theory associates an abelian group $\HF(Y,\mathfrak{t})$ to a closed, oriented Spin$^c$ 3-manifold $(Y,\mathfrak{t})$, and a homomorphism
$$F_{W,\mathfrak{s}}:\HF(Y_1,\mathfrak{t}_1)\to \HF(Y_2,\mathfrak{t}_2)$$ to a Spin$^c$ cobordism $(W,\mathfrak{s})$ between two Spin$^c$ 3-manifolds $(Y_1,\mathfrak{t}_1)$ and $(Y_2,\mathfrak{t}_2)$.  Write $\HF(Y)$ for the direct sum $\oplus_{\mathfrak{t}}\HF(Y,\mathfrak{t})$ over all Spin$^c$ structures $\mathfrak{t}$ on $Y$ and $F_W$ for the sum $\sum_{\mathfrak{s}}F_{W,\mathfrak{s}}$ over all Spin$^c$ structures $\mathfrak{s}$ on $W$.  In \cite{OSzContact}, Ozsv\'ath and Szab\'o introduced an invariant $c(Y, \xi)\in \HF(-Y)$ for closed contact 3-manifold $(Y,\xi)$.  If the contact manifold $(Y_K, \xi_K)$ is obtained from $(Y,\xi)$ by contact $(+1)$-surgery along a Legendrian knot $K$, then we have
\begin{equation}\label{functorial}
F_{-W}(c(Y, \xi))=c(Y_K, \xi_K),
\end{equation}
where $-W$ stands for the cobordism induced by the surgery with reversed orientation.
This functorial property of the contact invariant can be proved by an adaption of \cite[Theorem 4.2]{OSzContact} (cf. \cite[Theorem 2.3]{ls2}).

Next, suppose $L=L_1\cup L_2$ is an oriented Legendrian two-component link in $(S^3,\xi_{std})$, and the linking number of $L_1$ and $L_2$ is $l$. The resulting contact $(+1)$-surgery along $L$ is denoted by $(S^3_{\Lambda}(L),\xi_L)$, where $$\Lambda=\left(
 \begin{array}{cc}   tb(L_1)+1 & l \\  l & tb(L_2)+1
 \end{array}
\right) $$
is the topological surgery framing matrix.
Let $W$ be the cobordism from $S^{3}$ to $S^{3}_{\Lambda}(L) $ induced by this surgery, and $-W$ be $W$ with reversed orientation. This gives a map $$F_{-W}:\HF(S^3)\rightarrow \HF(S^{3}_{-\Lambda}(\overline{L})).$$
In particular,
the contact invariants $$c(S^3, \xi_{std})\in \HF(-S^3)=\HF(S^3)$$ and  $$c(S^{3}_{\Lambda}(L), \xi_L)\in \HF(-S^{3}_{\Lambda}(L))=\HF(S^{3}_{-\Lambda}(\overline{L}))$$ are related by
\begin{equation} \label{ContactInvMap}
F_{-W}(c(S^3, \xi_{std}))=c(S^{3}_{\Lambda}(L), \xi_L).
\end{equation}
from the functoriality (\ref{functorial}) and the composition law \cite[Theorem 3.4]{OSzFour}.

\medskip
In \cite{g}, Golla investigated the contact invariant of a contact manifold given by contact surgery along a Legendrian knot in $(S^3,\xi_{std})$. In particular, by
\cite[Theorem 1.1]{g}, the contact 3-manifold obtained by contact $(+1)$-surgery along the Legendrian knot $L_i$ ($i=1,2$) in $(S^3, \xi_{std})$ has nonvanishing contact invariant if and only if $L_{i}$ satisfies the following three conditions:
\begin{equation}\label{Golla1}
tb(L_i)=2\tau(L_i)-1,
\end{equation}
\begin{equation} \label{Golla2}
rot(L_i)=0,
\end{equation}
\begin{equation}\label{Golla3}
\tau(L_i)=\nu(L_i).
\end{equation}

Hence, if either $L_1$ or $L_2$ does not satisfy one of these three conditions, then it follows readily from the functoriality (\ref{functorial}) that
the contact invariant $c(S^3_{\Lambda}(L),\xi_L)$ must vanish as well.

\begin{remark}\label{maxtb}
There exists a two-component link such that the knot type of each component has a Legendrian representative satisfying  the above three conditions, but the link type of this
two-component link has no
Legendrian representative with both two components satisfying the above three conditions simultaneously \cite[Section 5.6]{e}.
\end{remark}

\section{Link surgery formula for two-component links}

In this section, we recall the link surgery formula for two-component links developed in \cite{mo}.  The link surgery formula given by Manolescu and Ozsv\'ath is a generalization of the knot surgery formula given by Ozsv\'ath and Szab\'o \cite{OSzSurgery}\cite{OSzRatSurg}.  The idea is to compute the Heegaard Floer homology of a 3-manifold obtained by surgery along a knot in terms of a mapping cone construction, and more generally surgery along a link in terms of the hyperbox of the link surgery complex.  In addition, the cobordism map $F_{-W}$ is realized as the induced map of the inclusion of complexes in the system of hyperboxes \cite[Theorem 14.3]{mo}.

We go over the construction for two-component links. Suppose $L$ is an oriented link with two components $K_1$ and $K_2$, and the linking number of $K_1$ and $K_2$ is $l$. Given the topological surgery framing matrix $$\Lambda=\left(
 \begin{array}{cc}   p_1 & l \\  l & p_2
 \end{array}
\right),$$ we will see the computation of $\HF(S^{3}_{\Lambda}(L))$.

Let $\mathbb{H}(L)_{i}=\frac{l}{2}+\mathbb{Z}, \, i=1,2.$  Define the affine lattice $\mathbb{H}(L)$ over $H_{1}(S^{3}-L)\cong \mathbb{Z}^{2}$ by
$$\mathbb{H}(L)=\mathbb{H}(L)_{1}\oplus\mathbb{H}(L)_{2}.$$  Here we identify $H_1(S^3-L)$ with $\Z^2$ using the oriented meridians of the components as the generators.  The elements of $\mathbb{H}(L)$ correspond to Spin$^{c}$ structures on $S^3$ relative to $L$. 
Also, let $$\mathbb{H}(K_{i})=\mathbb{Z}; \;\; \mathbb{H}(\emptyset)=\{0\}.$$  The elements of  $\mathbb{H}(K_{i})$  correspond to Spin$^{c}$ structures on $S^3$ relative to $K_i$ for $i=1,2$, while $0\in \mathbb{H}(\emptyset)$ corresponds to the unique Spin$^{c}$ structure on $S^3$.
Furthermore, let $+K_i$ and $-K_i$ represent the component $K_i$ of $L$ with the induced and opposite orientation, respectively.
For $M=\epsilon_{1} K_1$ or $\epsilon_{2} K_2$, where $\epsilon_{i}$ is $+$ or $-$ for $i=1,2$,
define
\begin{align*}
\psi^{M}: \; \mathbb{H}(L) & \longrightarrow \mathbb{H}(L-M) \\
s=(s_{1}, s_{2}) & \longmapsto  s_{j}-\frac{lk(+K_{j}, M)}{2},
\end{align*}
where $+K_{j}=L-M$ denotes the component of $L$ other than $M$. 
We also define
\begin{align*}
\psi^{ \epsilon_{1}K_1\cup \epsilon_{2}K_2} : \mathbb{H}(L) &\longrightarrow  \mathbb{H}(\emptyset) \\
 s &\longmapsto 0.
\end{align*}

Now we consider Spin$^{c}$ structures on $S^{3}_{\Lambda}(L)$.  Let $H(L,\Lambda)$ be the (possibly degenerate) sublattice of $\mathbb{Z}^2$ generated by the two columns $\Lambda_1$ and $\Lambda_2$ of $\Lambda$.  We denote the quotient of $s\in \mathbb{H}(L)$ in $\mathbb{H}(L)/H(L,\Lambda)$ by $[s]$.  For any $\mathfrak{u} \in \spin(S^{3}_{\Lambda}(L))$, there is a standard way of associating an element $[s] \in \mathbb{H}(L)/H(L,\Lambda)$.  This gives an identification of the set $\mathbb{H}(L)/H(L,\Lambda)$  with $\spin(S^{3}_{\Lambda}(L))$.


For any $s\in\mathbb{H}(L)$, and any choice of $\epsilon_1, \epsilon_2\in \{\pm\}$, there is a square of chain complexes:

 $$\xymatrix{
  \hat{\mathfrak{A}}(\mathcal{H}^{K_1},\psi^{\epsilon_{2}K_2}(s))\ar[rr]^{\Phi_{\psi^{\epsilon_{2}K_2}(s)}^{\epsilon_{1}K_{1}}}&  & \hat{\mathfrak{A}}(\mathcal{H}^{\emptyset},\psi^{\epsilon_{1}K_1\cup\epsilon_{2}K_2}(s))\\
    \hat{\mathfrak{A}}(\mathcal{H}^{L},s) \ar[rr]_{\Phi_{s}^{\epsilon_{1}K_{1}}}\ar[u]^{\Phi_{s}^{\epsilon_{2}K_{2}}}\ar[urr]_{\Phi_{s}^{\epsilon_{1}K_1\cup\epsilon_{2}K_2}} & & \hat{\mathfrak{A}}(\mathcal{H}^{K_2},\psi^{\epsilon_{1}K_1}(s))\ar[u]_{\Phi_{\psi^{\epsilon_{1}K_1}(s)}^{\epsilon_{2}K_{2}}}
      }$$


Here, $\hat{\mathfrak{A}}(\mathcal{H}^{L},s)$,  $\hat{\mathfrak{A}}(\mathcal{H}^{K_1},\psi^{\epsilon_{2}K_2}(s))$,  $ \hat{\mathfrak{A}}(\mathcal{H}^{K_2},\psi^{\epsilon_{1}K_1}(s))$, and $\hat{\mathfrak{A}}(\mathcal{H}^{\emptyset},\psi^{\epsilon_{1}K_1\cup \epsilon_{2}K_2}(s))$ are generalized Floer complexes that can be determined from a given Heegaard diagram of the link $L$. The edge maps $\Phi$'s between these generalized Floer complexes are defined by counting holomorphic polygons with certain properties in the Heegaard diagram. The diagonal map $\Phi_{s}^{\epsilon_{1}K_1\cup\epsilon_{2}K_2}$ is a chain homotopy equivalence between  $\Phi_{\psi^{\epsilon_{1}K_1}(s)}^{\epsilon_{2}K_{2}}\circ\Phi_{s}^{\epsilon_{1}K_{1}}$ and $\Phi_{\psi^{\epsilon_{2}K_2}(s)}^{\epsilon_{1}K_{1}}\circ \Phi_{s}^{\epsilon_{2}K_{2}}$.
This is a unified expression of the four squares in \cite[Page 20]{l1}.

Following the notations of \cite[Sections 4]{mo}, we denote
$$C_s^{00}=\hat{\mathfrak{A}}(\mathcal{H}^{L},s), \;\;   C_s^{10}=\hat{\mathfrak{A}}(\mathcal{H}^{K_2},\psi^{+K_1}(s)),$$
$$C_s^{01}=\hat{\mathfrak{A}}(\mathcal{H}^{K_1},\psi^{+K_2}(s)), \;\; C_s^{11}=\hat{\mathfrak{A}}(\mathcal{H}^{\emptyset},\psi^{+K_1\cup +K_2}(s)). $$
Fix $\mathfrak{u} \in \spin(S^{3}_{\Lambda}(L))$. The \emph{link surgery formula} for two-component links is a hyperbox of complexes $(\hat{\mathcal{C}}, \hat{\mathcal{D}}, \mathfrak{u})$ shown as follow:

$$\xymatrix{
  \prod\limits_{s\in \mathbb{H}(L), [s]=\mathfrak{u}}C_s^{01}\ar[rr]&
  & \prod\limits_{s\in \mathbb{H}(L), [s]=\mathfrak{u}}C_s^{11}\\
    \prod\limits_{s\in \mathbb{H}(L), [s]=\mathfrak{u}}C_s^{00} \ar[rr]\ar[u]
    \ar[urr]
    & & \prod\limits_{s\in \mathbb{H}(L), [s]=\mathfrak{u}}C_s^{10}\ar[u]  }$$


Here, the horizontal arrows consist of maps of the form $\Phi_s^{\epsilon_1 K_1}$ or $\Phi_{\psi^{\epsilon_{2}K_2}(s)}^{\epsilon_{1}K_{1}}$, the vertical arrows of maps $\Phi_s^{\epsilon_2 K_2}$ or $\Phi_{\psi^{\epsilon_{1}K_1}(s)}^{\epsilon_{2}K_{2}}$, and the diagonal of maps $\Phi_s^{\epsilon_1 K_1\cup \epsilon_2 K_2}$.  The value of $s$ in the targets of each map are shifted by an amount depending on the type of map and the framing $\Lambda$: whenever we have a negatively oriented component $-K_i$ in the superscript of a map $\Phi_s$, we add the vector $\Lambda_i$ to $s$.  So, for example, the maps $\Phi_{s}^{-K_1}$ and $\Phi_{s}^{-K_2}$ shift $s$ by $(p_1, l)$ and $(l, p_2$), respectively, and $\Phi_{s}^{-K_1\cup -K_2}$ shifts $s$ by $(p_1+l, l+p_2)$.  We refer the reader to \cite[Sections 4, 5, 8 and 9]{mo}, \cite[Section 2]{li1}, or \cite[Section 4]{l1} for details.

In Figure~\ref{fig: E}, we exhibit a more concrete representation of $(\hat{\mathcal{C}}, \hat{\mathcal{D}}, \mathfrak{u})$ for which a square is drawn at the lattice point $s=(s_1, s_2)$ with $[s]=\mathfrak{u}$, and the complexes 
$C_s^{00}$, $C_s^{01}$, $C_s^{10}$ and $C_s^{11}$ are placed at the lower left, the upper left, the lower right and the upper right corner of the square, respectively.  Note that while $\Phi_{s}^{+K_{1}}$ and $\Phi_{s}^{+K_{2}}$ stay at the original lattice point, $\Phi_{s}^{-K_{1}}$ and $\Phi_{s}^{-K_{2}}$ map to the complexes at the lattice points $(s_1+p_1, s_2+l)$ and $(s_1+l, s_2+p_2$), respectively.

\begin{figure}[!ht]
\begin{overpic}
{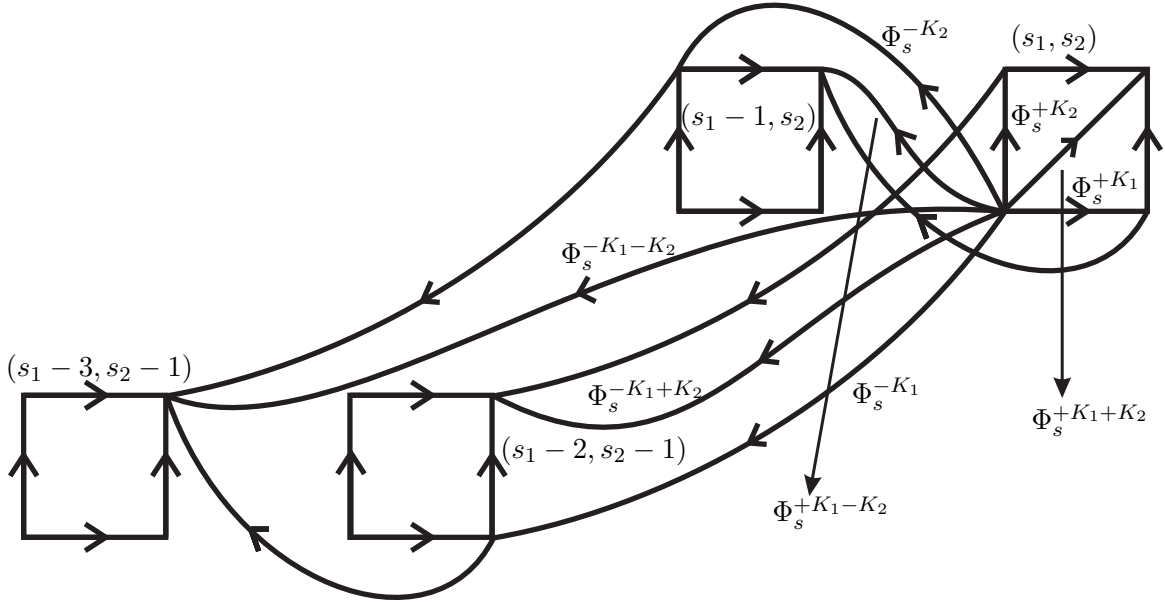}
\put(403, 153){\small{$\Phi_{s}^{+K_{1}}$}}
\put(331, 210){\small{$\Phi_{s}^{-K_{2}}$}}
\put(380, 180){\small{$\Phi_{s}^{+K_{2}}$}}
\put(320, 75){\small{$\Phi_{s}^{-K_{1}}$}}
\put(210, 128){\small{$\Phi_{s}^{-K_{1}-K_{2}}$}}
\put(220, 74){\small{$\Phi_{s}^{-K_{1}+K_{2}}$}}
\put(290, 30){\small{$\Phi_{s}^{+K_{1}-K_{2}}$}}
\put(388, 65){\small{$\Phi_{s}^{+K_{1}+K_{2}}$}}
\put(380, 210){\small{$(s_{1}, s_{2})$}}
\put(255, 180){\small{$(s_{1}-1, s_{2})$}}
\put(0, 85){\small{$(s_{1}-3, s_{2}-1)$}}
\put(187, 54){\small{$(s_{1}-2, s_{2}-1)$}}

\end{overpic}
\caption[]{Four squares of a hyperbox of chain complexes and maps between them.  The surgery framing matrix is $\left(\begin{array}{cc}   -2 & -1 \\  -1 & 0 \end{array}\right)$.}

\label{fig: E}
\end{figure}

It is often convenient to study $(\hat{\mathcal{C}}, \hat{\mathcal{D}}, \mathfrak{u})$ by introducing a filtration and consider the associated spectral sequence. Here, we define the filtration $\mathcal{F}(x)$ to be the number of components of $L'\subset L$ if $x\in \hat{\mathfrak{A}}(\mathcal{H}^{L'})$. Thus, the complex at the lower left corner of each square has filtration level 2; the complex at the lower right or the upper left corner of each square has filtration level 1; and the complex at the upper right corner of each square has filtration level 0.  Since the largest difference in the filtration levels is 2, the $k^{th}$ differential in the spectral sequence, $\partial_k$, mush vanish for $k>2$.  According to \cite[Section 3]{li1}, the associated spectral sequence has  $$E_{0}=(\hat{\mathcal{C}}, \partial_{0}),$$ $$E_{1}=(H_{\ast}(\hat{\mathcal{C}}, \partial_{0}), \partial_{1}),$$ and $$\HF(S^{3}_{\Lambda}(L), \mathfrak{u})=E_{\infty}=H_{\ast}(E_{2})=H_{\ast}(H_{\ast}(H_{\ast}(\hat{\mathcal{C}}, \partial_{0}), \partial_{1}), \partial_{2}).$$

Let us explain the $E_{1}$ page of the surgery chain complex in greater detail. Figure~\ref{fig: E1} exhibits a typical example of an $E_1$ page associated to a 2-dimensional hyperbox $(\hat{\mathcal{C}}, \hat{\mathcal{D}}, \mathfrak{u})$.
Observe that $\partial_{0}$ is the internal differential of each generalized Floer complex.  Hence we have $H_{\ast}(\hat{\mathfrak{A}}(\mathcal{H}^{L},s))$ at the lower left corner of the square at the lattice point $s=(s_1, s_2)$, which turns out to be isomorphic to $\HF$ of a large surgery along $L$ in a certain Spin$^{c}$ structure.  Similarly, $H_{\ast}(\hat{\mathfrak{A}}(\mathcal{H}^{K_1},\psi^{+K_2}(s)))$ at the upper left corner and $H_{\ast}(\hat{\mathfrak{A}}(\mathcal{H}^{K_2},\psi^{+K_1}(s)))$ at the lower right corner of each square are isomorphic to $\HF$ of large surgeries along $K_1$ and $K_2$ in certain Spin$^{c}$ structures, respectively; and  $H_{\ast}(\hat{\mathfrak{A}}(\mathcal{H}^{\emptyset},\psi^{+K_1\cup +K_2}(s)))$ at the upper right corner of each square is isomorphic to $\HF(S^{3})$.

\begin{figure}[htb]
\begin{overpic}
{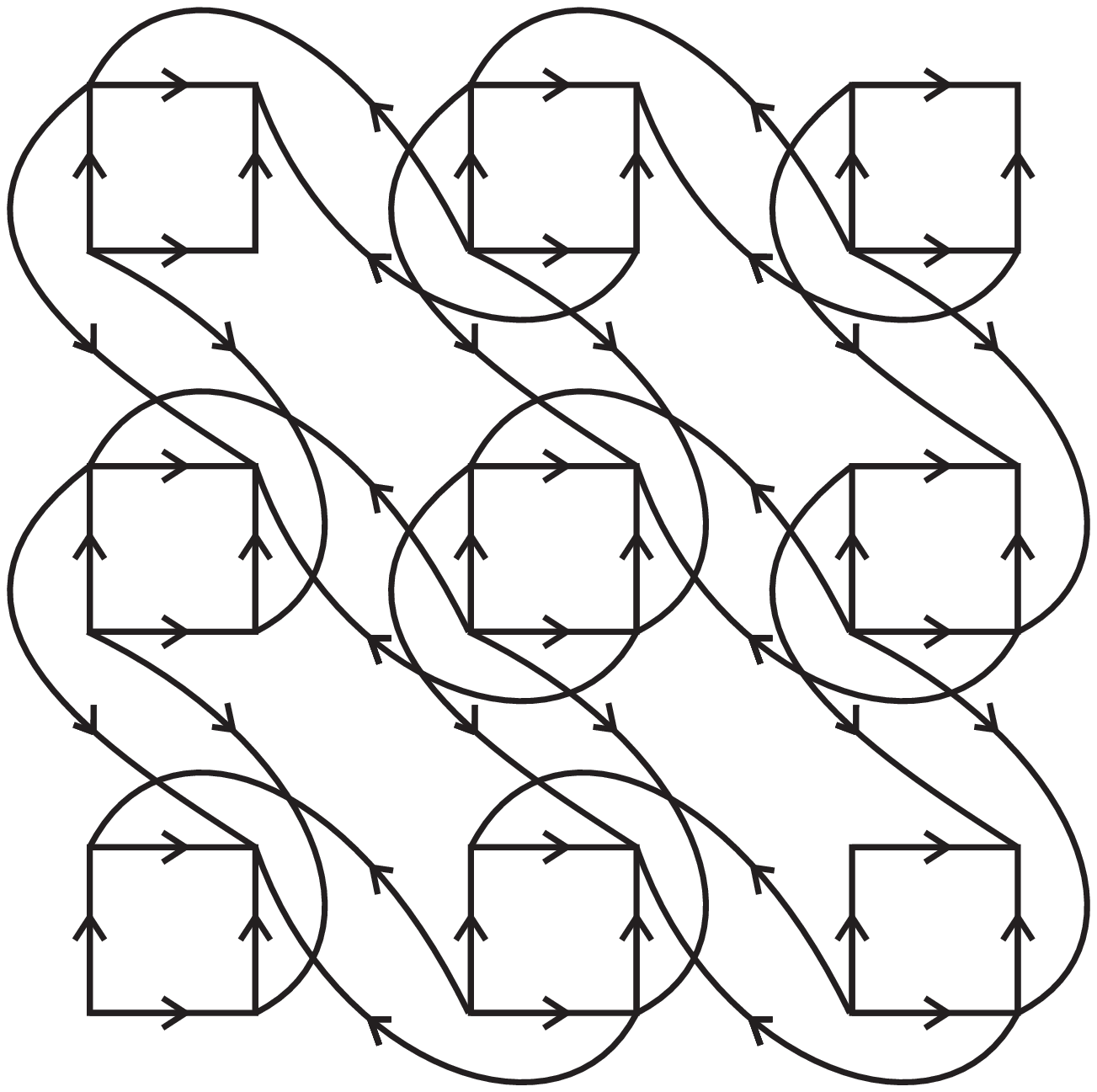}
\put(150, 212){$b^{1}$}
\put(219, 146){$b^{2}$}
\put(220, 212){$c$}
\put(161, 296){\small{$(s_1, s_2+1)$}}
\put(161, 163){\small{$(s_1, s_2)$}}
\put(161, 35){\small{$(s_1, s_2-1)$}}
\put(30, 163){\small{$(s_1-1, s_2)$}}
\put(292, 163){\small{$(s_1+1, s_2)$}}
\put(137, 109){\small{$\phi^{-K_1}_{\psi^{+K_2}(s)}$}}
\put(163, 199){\small{$\phi^{+K_1}_{\psi^{+K_2}(s)}$}}
\put(105, 136){\small{$\phi^{-K_2}_{\psi^{+K_1}(s)}$}}
\put(176, 178){\small{$\phi^{+K_2}_{\psi^{+K_1}(s)}$}}

\end{overpic}
\caption[]{Part of an $E_{1}$ page for the surgery framing matrix $\left(\begin{array}{cc}   0 & -1 \\  -1 & 0
 \end{array}\right)$ }

\label{fig: E1}
\end{figure}

Next, we consider the differential $\partial_1$. For $i=1,2$, let $\phi^{\epsilon_{i}K_i}_{s}$  be the homomorphism induced from $\Phi^{\epsilon_{i}K_i}_{s}$. Let $\phi^{\epsilon_{1}K_1}_{\psi^{+K_2}(s)}$ and $\phi^{\epsilon_{2}K_2}_{\psi^{+K_1}(s)}$ be the homomorphisms induced from $\Phi^{\epsilon_{1}K_1}_{\psi^{+K_2}(s)}$ and $\Phi^{\epsilon_{2}K_2}_{\psi^{+K_1}(s)}$, respectively.   Then $\partial_1$ consists of a collection of \emph{short edge maps} $\phi^{+K_1}$ and $\phi^{+K_2}$ that stay at the original lattice point, and another collection of \emph{long edge maps} $\phi^{-K_1}$ and $\phi^{-K_2}$ that shift the position by the vectors $(p_1, l)$ and $(l, p_2)$, respectively. The most relevant maps for our purposes are the ones that map into the homology $H_{\ast}(\hat{\mathfrak{A}}(\mathcal{H}^{\emptyset},\psi^{+K_1\cup +K_2}(s)))$ at the upper right corner of each square.  If we let $N$ be a sufficiently large integer, then under the above identification of $H_{\ast}(\hat{\mathcal{C}}, \partial_{0})$ with Heegaard Floer homology of large integer surgeries, we can identify the short edge map initiated from the upper left corner as $$\phi^{+K_1}_{\psi^{+K_2}(s)}:\HF(S^{3}_{N}(K_1), s_{1}-\frac{l}{2})\rightarrow\HF(S^{3}),$$
the long edge map initiated from the upper left corner as
$$\phi^{-K_1}_{\psi^{+K_2}(s)}:\HF(S^{3}_{N}(K_1), s_{1}-\frac{l}{2})\rightarrow\HF(S^{3}),$$

the short edge map initiated from the lower right corner as
$$\phi^{+K_2}_{\psi^{+K_1}(s)}:\HF(S^{3}_{N}(K_2), s_{2}-\frac{l}{2})\rightarrow\HF(S^{3}),$$
and the long edge map initiated from the lower right corner as
$$\phi^{-K_2}_{\psi^{+K_1}(s)}:\HF(S^{3}_{N}(K_2), s_{2}-\frac{l}{2})\rightarrow\HF(S^{3}),$$

Note that the targets $\HF(S^{3})$ of the first two maps are $H_{\ast}(\hat{\mathfrak{A}}(\mathcal{H}^{\emptyset},\psi^{+K_1\cup +K_2}(s)))$ and $H_{\ast}(\hat{\mathfrak{A}}(\mathcal{H}^{\emptyset},\psi^{-K_1\cup +K_2}(s)))$, respectively, and  the targets $\HF(S^{3})$ of the last two maps are $H_{\ast}(\hat{\mathfrak{A}}(\mathcal{H}^{\emptyset}, \psi^{+K_1\cup +K_2}(s)))$ and $H_{\ast}(\hat{\mathfrak{A}}(\mathcal{H}^{\emptyset},$ $\psi^{+K_1\cup -K_2}(s)))$, respectively.
According to \cite[Theorem 14.3]{mo}, the $p_i$-surgery on $S^3$ along $K_i$, $i=1,2$, corresponds to a 1-dimensional subcomplex in the 2-dimensional hyperbox $(\hat{\mathcal{C}}, \hat{\mathcal{D}}, \mathfrak{u})$.  Thus, by \cite[Remark 3.23]{mo}, the maps $\phi^{+K_1}_{\psi^{+K_2}(s)}$ and $\phi^{-K_1}_{\psi^{+K_2}(s)}$ are equivalent to the \emph{vertical} and \emph{horizontal} maps $\hat{v}_{K_1}$ and $\hat{h}_{K_1}$ defined in \cite{OSzSurgery}, respectively. The same thing holds for $\phi^{+K_2}_{\psi^{+K_1}(s)}$ and $\phi^{-K_2}_{\psi^{+K_1}(s)}$.

\section{Vanishing contact invariants}

\begin{proof}[Proof of Theorem \ref{thm: main}]
 It follows from $\nu^+(L_2)=\nu^+(\overline{L_2})=0$ and (\ref{equivalence}) that $\tau (L_2)=0$. We claim that it suffices to consider the case where the Thurston-Bennequin invariant  $tb(L_2)=-1$.  Otherwise, $tb(L_2)$ must be strictly less than $-1$ by the inequality $tb(L_2)+|rot(L_2)|\le 2\tau(L_2)-1=-1$ (\cite[Theorem 1]{Olga}), thus violating the condition of (\ref{Golla1}).  This then implies the triviality of the contact invariant by our discussion near the end of Section 2.

We first treat the case where $L_2$ is a Legendrian unknot. We try to determine the contact invariant $c(S^{3}_{\Lambda}(L), \xi_L)\in \HF(-S^{3}_{\Lambda}(L))=\HF(S^3_{-\Lambda}(\overline{L}))$.  Note that $c(S^3, \xi_{std})$ is the unique generator of $\HF(S^3)\cong \F$. Hence  by (\ref{ContactInvMap}), $c(S^{3}_{\Lambda}(L), \xi_L)$ is the image of the generator under the cobordism map $F_{-W}:\HF(S^3)\rightarrow \HF(S^3_{-\Lambda}(\overline{L}))$, so $c(S^{3}_{\Lambda}(L), \xi_L)=0$ is equivalent to $F_{-W}$ being the zero map.

We resort to \cite[Theorem 14.3]{mo} to understand this map, which identifies $F_{-W}$ with the induced map of the inclusion $$\prod\limits_{s\in \mathbb{H}(\overline{L}), [s]=\mathfrak{u}}\hat{\mathfrak{A}}(\mathcal{H}^{\emptyset},\psi^{\overline{L_1}\cup \overline{L_2}}(s))\hookrightarrow (\hat{\mathcal{C}}, \hat{\mathcal{D}}, \mathfrak{u}).$$
In order to prove that $F_{-W}$ vanishes, it suffices to show that for each $s=(s_1, s_2)\in\mathbb{H}(\overline{L})$, $c_{s_1, s_2}$, the generator of $\HF(S^3)$ at the upper right corner of the square at the lattice point $(s_1, s_2)$,  is a boundary in the $E_1$ page of the spectral sequence (or equivalently, trivial in the $E_2$ page).

For the subsequent argument, we will still refer to Figure \ref{fig: E1} for a schematic picture of the $E_1$ page of the spectral sequence, although we should point out that at present the surgery is performed along the link $\overline{L}=\overline{L_1}\cup \overline{L_2}$, $\overline{L_1}$ and $\overline{L_2}$ correspond to $K_1$ and $K_2$ in Figure
\ref{fig: E1}, respectively, and the topological surgery framing matrix is $$-\Lambda=\left(
 \begin{array}{cc}   -(tb(L_1)+1) & -l \\  -l & 0
 \end{array}
\right).$$

As earlier, we use $N$ to denote a sufficiently large integer.  Since $\overline{L_2}$ is the unknot, the homology group that was identified with $\HF(S^{3}_{N}(\overline{L_2}), s_{2}+\frac{l}{2})$ at the lower right corner of the square at each lattice point $(s_1, s_2)$ is 1-dimensional.  We denote the generator of the homology group by $b^2_{s_1, s_2}$. Clearly, we have:

(1) If $s_{2}+\frac{l}{2}>0$, then $\phi^{+\overline{L_2}}_{\psi^{\overline{L_1}}(s)}$ is an isomorphism, and $\phi^{-\overline{L_2}}_{\psi^{\overline{L_1}}(s)}$  is the trivial map. So
\begin{equation*}\label{(1)}
\partial_{1} b^{2}_{s_{1}, s_{2}}=c_{s_{1}, s_{2}}.
\end{equation*}

(2) If $s_{2}+\frac{l}{2}<0$, then $\phi^{+\overline{L_2}}_{\psi^{\overline{L_1}}(s)}$ is the trivial map, and  $\phi^{-\overline{L_2}}_{\psi^{\overline{L_1}}(s)}$  is an isomorphism. So
\begin{equation*}\label{(2)}
\partial_{1} b^{2}_{s_{1}+l, s_{2}}=c_{s_{1}, s_{2}}.
\end{equation*}

(3) If $s_{2}+\frac{l}{2}=0$, then both $\phi^{+\overline{L_2}}_{\psi^{\overline{L_1}}(s)}$ and  $\phi^{-\overline{L_2}}_{\psi^{\overline{L_1}}(s)}$  are  isomorphisms. So
 \begin{equation*}\label{(3)}
 \partial_{1} b^{2}_{s_{1}, -\frac{l}{2}}=c_{s_{1}, -\frac{l}{2}}+c_{s_{1}-l, -\frac{l}{2}}.
\end{equation*}

On the other hand, we understand the general properties of the maps $\phi^{\pm\overline{L_1}}_{\psi^{\overline{L_2}}(s)}$ well when $s_1$ is sufficiently large:
In that case, the homology group $\HF(S^{3}_{N}(\overline{L_1}), s_{1}+\frac{l}{2})\cong \F$, and $\phi^{+\overline{L_1}}_{\psi^{\overline{L_2}}(s)}$ is an isomorphism while $\phi^{-\overline{L_1}}_{\psi^{\overline{L_2}}(s)}$ is the trivial map \cite{OSzSurgery}. Thus, if we denote the generator of the homology group at the upper left corner of the square at the lattice point $(s_1,-\frac{l}{2})$ by $b^{1}_{s_{1}, -\frac{l}{2}}$, then
\begin{equation}\label{(4)}
\partial_{1} b^{1}_{s_{1}, -\frac{l}{2}}=c_{s_{1}, -\frac{l}{2}}, \; \text{when} \; s_1 \gg 0.
\end{equation}

Let us put them together.  When $s_2\neq -\frac{l}{2}$, we can immediately see from Claims (1) and (2) that $c_{s_1, s_2}$ lies in the image of $\partial_1$. When $s_2=-\frac{l}{2}$, we can use Claim (3) and (\ref{(4)}) to find an explicit element $b$ such that $\partial_1(b)=c_{s_1, -\frac{l}{2}}$ under the assumption that the linking number $l$ is nonzero. More precisely, one can check that
$$\partial_1( b^{2}_{s_{1}+l, -\frac{l}{2}}+b^{2}_{s_{1}+2l, -\frac{l}{2}}+\cdots + b^{2}_{s_{1}+nl, -\frac{l}{2}}+b^1_{s_{1}+nl, -\frac{l}{2}})=c_{s_{1}, -\frac{l}{2}}$$
for $n$ large enough and $l>0$; and
$$\partial_1( b^{2}_{s_{1}, -\frac{l}{2}}+b^{2}_{s_{1}-l, -\frac{l}{2}}+\cdots + b^{2}_{s_{1}-nl, -\frac{l}{2}}+b^1_{s_{1}-(n+1)l, -\frac{l}{2}})=c_{s_{1}, -\frac{l}{2}}$$
for $n$ large enough and $l<0$.  In either case, this proves that  $c_{s_1, s_2}$ lies in the image of $\partial_1$ for each $s=(s_1, s_2)$, thus implying the theorem for the special case when $L_2$ is a Legendrian unknot.

More generally, since $L_2$ satisfies $\nu^+(L_2)=\nu^+(\overline{L_2})=0$, we apply (\ref{equivalence}) and conclude that $CFK^\infty(\overline{L_2})$ is filtered chain homotopy equivalent to $CFK^\infty (U) \oplus A$ for some acyclic complex $A$.  Then, the above argument for the unknot case nearly extends verbatim to the general case, except that the homology group $\HF(S^{3}_{N}(\overline{L_2}), s_{2}+\frac{l}{2})$  may not necessarily be 1-dimensional.  Nevertheless, we noticed that only the existence of the elements $b^2_{s_1, s_2}$ that satisfy Claims (1), (2) and (3) was really needed for the above proof. This can be attained in our case by taking the generators $b^2_{s_1, s_2}$ from the $CFK^\infty (U)$ summand in the filtered chain homotopy equivalent complex $CFK^\infty (U) \oplus A$.  The rest of the proof carries over for the general case.
\end{proof}

\begin{remark}
Indeed, based on a slightly more involved diagram-chasing-type argument like above, one can show that in general cases there exists $s=(s_1,s_2)\in\mathbb{H}(\overline{L})$ such that $c_{s_1, s_2}$ is nontrivial in the $E_2$ page if and only if $\tau(L_i)=\nu(L_i)$ and $\tau(\overline{L_i})=\nu(\overline{L_i})-1$ for $i=1,2$, under the assumption that the linking number is nonzero and $tb(L_i)=2\tau(L_i)-1$.  A better understanding of the higher differential $\partial_2$ in the $E_2$ page of the spectral sequence may lead to nonvanishing results of contact invariant.
\end{remark}

As a corollary, we show that contact $\frac{1}{n}$-surgery on $\sharp^{k}(S^{1}\times S^{2}, \xi_{t})$ along a homologically essential Legendrian knot $L$ yields a contact 3-manifold with vanishing contact invariant for any positive integer $n$, as claimed in Corollary \ref{cor}.

\begin{proof}[Proof of Corollary \ref{cor}]
The contact 3-manifold $\sharp^{k}(S^{1}\times S^{2},\xi_{t})$ can be obtained by contact $(+1)$-surgery on $(S^{3}, \xi_{std})$ along a Legendrian $k$-component unlink $L_{0}$. There exists a Legendrian knot $\tilde{L}$ in $(S^{3}, \xi_{std})$ which becomes the Legendrian knot $L$ in $\sharp^{k}(S^{1}\times S^{2},\xi_{t})$ after the contact $(+1)$-surgery along $L_0$. To find such an $\tilde{L}$, it suffices to perform Legendrian surgery on $\sharp^{k}(S^{1}\times S^{2},\xi_{t})$ along a Legendrian $k$-component link, each component of which lies in a  summand $(S^{1}\times S^{2},\xi_{t})$ and is disjoint from $L$, so that the result is $(S^{3}, \xi_{std})$. Then the image of $L$ in $(S^{3}, \xi_{std})$ is the desired $\tilde{L}$.

Note that contact $\frac{1}{n}$-surgery on $\sharp^{k}(S^{1}\times S^{2},\xi_{t})$ along  $L$ is equivalent to contact $(+1)$-surgery on  $\sharp^{k}(S^{1}\times S^{2},\xi_{t})$ along $n$ Legendrian push-offs of $L$ for any positive integer $n$.  Therefore, contact $\frac{1}{n}$-surgery on $\sharp^{k}(S^{1}\times S^{2},\xi_{t})$  along $L$ is equivalent to contact $(+1)$-surgery on $(S^{3}, \xi_{std})$ along a Legendrian $(k+n)$-component link $L'$, which is the union of the aforementioned Legendrian $k$-component unlink $L_{0}$ and $n$ Legendrian push-offs of the Legendrian knot $\tilde{L}$. See Figure~\ref{fig: lk3} for an example.

\begin{figure}[htb]
\begin{overpic}
{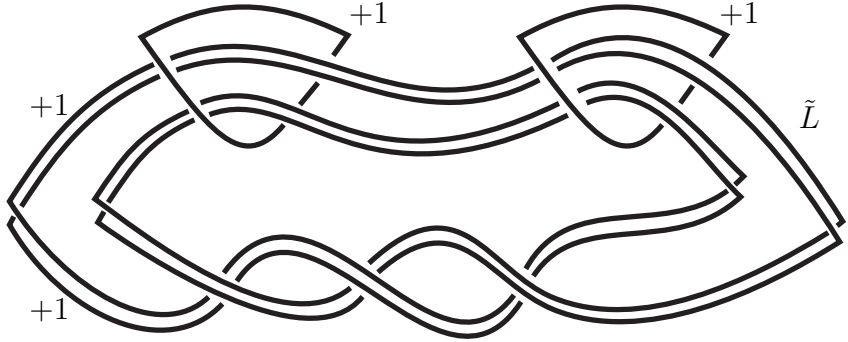}
\put(270, 120){$+1$}
\put(130, 120){$+1$}
\put(9, 85){$+1$}
\put(9, 8){$+1$}
\put(300, 80){$\tilde{L}$}

\end{overpic}
\caption{Contact $\frac{1}{2}$-surgery on $\sharp^{2}(S^{1}\times S^{2},\xi_{t})$ along a Legendrian knot $L$ is equivalent to contact $(+1)$-surgery on $(S^{3}, \xi_{std})$ along a Legendrian $4$-component link.}

\label{fig: lk3}
\end{figure}

If $L$ is not null-homologous in $\sharp^{k}(S^{1}\times S^{2})$, then the linking number of $\tilde{L}$ and one component of $L_0$ is nonzero. By Theorem~\ref{thm: main}, contact $(+1)$-surgery along the Legendrian two-component sublink of $L'$ formed by $\tilde{L}$ and that component of $L_0$ yields a contact 3-manifold with vanishing contact invariant.  Hence, it follows from (\ref{functorial}) that contact $(+1)$-surgery along the Legendrian $(k+n)$-component link $L'$ yields a contact 3-manifold with vanishing contact invariant as well.  This finishes the proof of the corollary.
\end{proof}

\section{Contact $(+1)$-surgeries along Legendrian Whitehead links}
\label{Whitehead}

In Figure~\ref{fig: wh1}, we draw a Legendrian two-component link in $(S^3,\xi_{std})$ with each component having Thurston-Bennequin invariant $-1$.  Denote the topological link types of this link and the mirror of it by $Wh$ and $\overline{Wh}$, respectively. A link in $S^3$ is called a \emph{Whitehead link} if it is of type $Wh$ or $\overline{Wh}$.

\begin{figure}[htb]
\begin{overpic}
{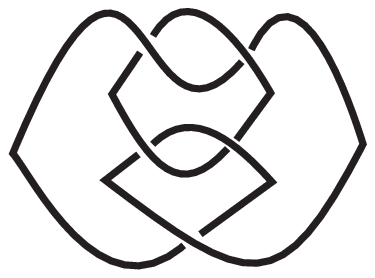}
\put(0, 65){$+1$}
\put(50, 80){$+1$}

\end{overpic}
\caption{A Legendrian Whitehead link with each component having Thurston-Bennequin invariant $-1$. The underlying topological type of this link is $Wh$.}

\label{fig: wh1}
\end{figure}


\begin{proof}[Proof of Proposition~\ref{PropWhitehead}]

First, we consider contact $(+1)$-surgery on $(S^3,\xi_{std})$ along a Legendrian representative $L$ of type $\overline{Wh}$.  According to \cite[Section 5.6]{e}, the sum of Thurston-Bennequin invariants of the two components of $L$ does not exceed $-5$.  Consequently, the Thurston-Bennequin invariant of one of the components of $L$ must be strictly less than $-1$.
By \"{O}zba\u{g}ci \cite[Theorem 3]{oz}, contact $(+1)$-surgery along a Legendrian unknot with Thurston-Bennequin invariant less than $-1$ yields an overtwisted contact 3-manifold.  Subsequently, the main result of Wand \cite{w} implies that contact $(+1)$-surgery along $L$ yields an overtwisted contact 3-manifold,  and must have vanishing contact invariant.

Now let $L=L_1\cup L_2$ be a Legendrian Whitehead link of type $Wh$. We only need to consider the case where both $L_1$ and $L_2$ have Thurston-Bennequin invariant $-1$. The result of contact $(+1)$-surgery along $L$ is denoted by $(S^3_{\mathbf{0}}(L), \xi_L)$, where
$$\mathbf{0}=\left(
 \begin{array}{cc}   0 & 0 \\  0 & 0
 \end{array}
\right) $$
is the topological surgery framing matrix.
We need to show that the contact invariant  $c(S^3_{\mathbf{0}}(L), \xi_L)$ vanishes.

Note that $\overline{L}=\overline{L_1}\cup \overline{L_2}$ and $-S^3_{\mathbf{0}}(L)=S^{3}_{\mathbf{0}}(\overline{L})$. Contact $(+1)$-surgery along the Legendrian unknot $L_{2}$ yields the unique tight contact structure $\xi_{t}$ on $S^{1}\times S^{2}$.  The subsequent 2-handle addition along $L_1$ yields $S^3_{\mathbf{0}}(L)$, and induces a homomorphism $$F_{1}: \HF(-S^1\times S^2)\longrightarrow \HF(-S^3_{\mathbf{0}}(L))$$ that sends the nontrivial contact invariant $c(S^1\times S^2,\xi_t)$ to $c(S^3_{\mathbf{0}}(L), \xi_L)$.

Recall that $$\HF(-S^1\times S^2)\cong\HF(S^{1}\times S^{2})\cong\mathbb{F}_{(-\frac{1}{2})}\oplus\mathbb{F}_{(\frac{1}{2})},$$ where the subscripts denote the absolute gradings,
and the contact invariant $c(S^1\times S^2,\xi_t)$ is supported in degree $-\frac{1}{2}$.
According to  \cite[Proposition 6.9]{l1}, $$\HF(S^3_{\mathbf{0}}(L))=\HF(S^{3}_{\mathbf{0}}(L), (0,0))\cong \mathbb{F}_{(0)}\oplus\mathbb{F}_{(0)}\oplus\mathbb{F}_{(-1)}\oplus\mathbb{F}_{(-1)},$$ where $(0,0)$ denotes the torsion Spin$^{c}$ structure on $S^3_{\mathbf{0}}(L)$.  So $$\HF(-S^3_{\mathbf{0}}(L))\cong \mathbb{F}_{(0)}\oplus\mathbb{F}_{(0)}\oplus\mathbb{F}_{(1)}\oplus\mathbb{F}_{(1)}.$$
Since $L_1$ is null-homologous in $S^1\times S^2$,  the homomorphism $F_1$ shifts the absolute degree by $-\frac{1}{2}$ \cite[Lemma 3.1]{OSzAbGr}. As $c(S^3_{\mathbf{0}}(L), \xi_L)=F_1(c(S^1\times S^2,\xi_t))$ and there are no nonzero elements in $\HF(-S^3_{\mathbf{0}}(L))$ supported in grading $-1$, we conclude that $c(S^3_{\mathbf{0}}(L), \xi_L)=0$.  This finishes the proof.
\end{proof}

\begin{remark}
Indeed, Theorem~\ref{overtwisted} implies that contact $(+1)$-surgery along the Legendrian Whitehead link shown in Figure~\ref{fig: wh1} yields an overtwisted contact 3-manifold. On the other hand, it is still unknown to date whether all Legendrian Whitehead links with each component having Thurston-Bennequin invariant $-1$ are Legendrian isotopic or not.  Hence, the obvious argument cannot be applied here to conclude that any contact 3-manifold obtained by contact $(+1)$-surgery along  a Legendrian Whitehead link is overtwisted.
\end{remark}

\section{Contact $(+1)$-surgeries yielding overtwisted contact 3-manifolds}
\label{Overtwisted}

\begin{proof}[Proof of Theorem \ref{overtwisted}]

We prove Theorem \ref{overtwisted} only in the case that the front projection contains the  configuration in Figure~\ref{fig: ot3}(a).  The same proof works for all other cases.


We construct a Legendrian knot $L'$ in $(S^{3}, \xi_{std})$ such that it can be divided into four segments $L'_1$, $L'_2$, $L'_3$ and $L'_4$. Two segments, $L'_3$ and $L'_4$, are contained in the dashed box in Figure~\ref{fig: ot2}.
For the other two segments, $L'_1$ is the downward Legendrian push-off of the part of $L_1$ outside the dashed box, and $L'_2$ is the upward Legendrian push-off of the part of $L_2$ outside the dashed box.

\begin{figure}[htb]
\begin{overpic}
{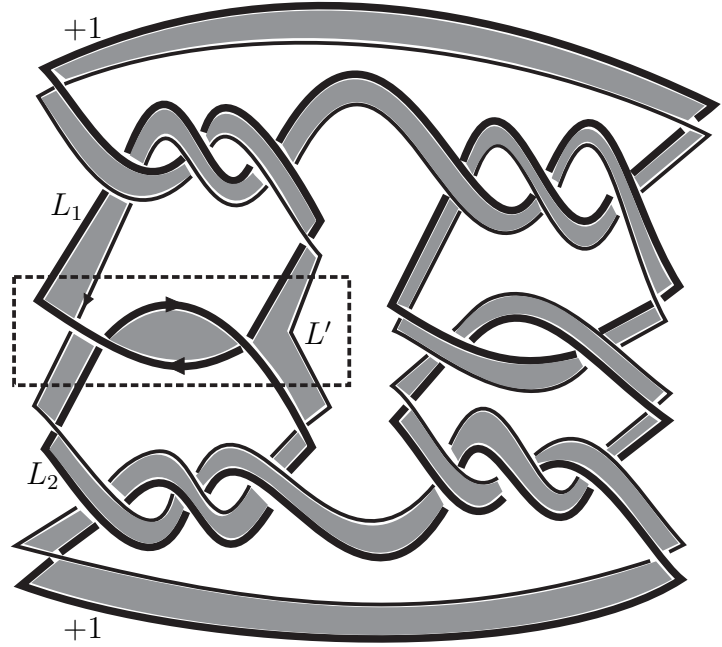}

\put(6, 60){$L_{2}$}
\put(15, 162){$L_{1}$}
\put(111, 113){$L'$}
\put(20, 229){$+1$}
\put(20, 2){$+1$}

\end{overpic}
\caption{An example of a contact $(+1)$-surgery diagram which satisfies the assumption of Theorem~\ref{overtwisted}.  The shaded area is a thrice-punctured sphere. The thin knot is $L'$. The two segments of $L'$ contained in the dashed box, $L'_3$ and $L'_4$, have no cusps and one cusp, respectively.}
\label{fig: ot2}
\end{figure}

There is a thrice-punctured sphere $S$ whose boundary consists of $L_1$, $L_2$ and $L'$.  See Figure~\ref{fig: ot2}. We orient $L_1$, $L_2$ and $L'$ as the boundary of $S$. Part of $S$ is contained in the dashed box in Figure~\ref{fig: ot2}. The part of $S$ outside the dashed box consists of two bands.  For brevity, we call the part of a knot in (resp. outside) the dashed box the inside part (resp. outside part) of the knot.

We compute the Thurston-Bennequin invariant of $L'$.

\begin{lemma} \label{ContactFraming} $tb(L')=tb(L_1)+tb(L_2)+2(l+1)$, where $l$ is the linking number $lk(L_1,L_2)$ of $L_1$ and $L_2$.
\end{lemma}

\begin{proof}
For the Legendrian knot $L'$, the Thurston-Bennequin invariant
\begin{equation}\label{tb}
tb(L')=w(L')-\frac{1}{2}c(L'),
\end{equation}
where $w(L')$ and $c(L')$ denote the writhe and the number of cusps of (the front projection of) $L'$, respectively.  Self-crossings of $L'$ consists of self-crossings of $L_1'$, self-crossings of $L_2'$ and crossings of $L_1'$ and $L_2'$. For $i=1,2$, self-crossings of $L'_i$ contribute $w(L_i)$ to $w(L')$. The crossings of $L'_1$ and $L'_2$ contribute $2(l+1)$ to $w(L')$. This can be seen as follows. The two crossings of $L_1$ and $L_2$ inside the dashed box contribute $-1$ to $lk(L_1, L_2)$.  So the crossings of $L_1$ and $L_2$ outside the dashed box contribute $l+1$ to $lk(L_1, L_2)$. Recall that $L'_{i}$ is a Legendrian push-off of the outside part of $L_{i}$ for $i=1,2$. Each crossing of $L'_1$ and $L'_2$ is induced by a crossing of the outside parts of $L_1$ and $L_2$. See Figure~\ref{fig: writhe} for all  possible configurations of $L'$ near a crossing of the outside parts of $L_{1}$ and $L_{2}$.  A crossing of the outside parts of $L_1$ and $L_2$ and the nearby crossing  of $L'_1$ and $L'_2$  have the same sign. So the number of crossings of $L'_1$ and $L'_2$, counted with sign, equals that of the outside parts of  $L_1$ and $L_2$, counted with sign, which is $2(l+1)$. Hence we have
\begin{equation}\label{wr}
w(L')=w(L_1)+w(L_2)+2(l+1).
\end{equation}
As $L'$ and $L_{1}\cup L_{2}$ have the same number of cusps,
\begin{equation}\label{cusp}
c(L')=c(L_1)+c(L_2).
\end{equation}
The lemma follows from  (\ref{tb}), (\ref{wr}) and (\ref{cusp}).
\end{proof}

\begin{figure}[htb]
\begin{overpic}
{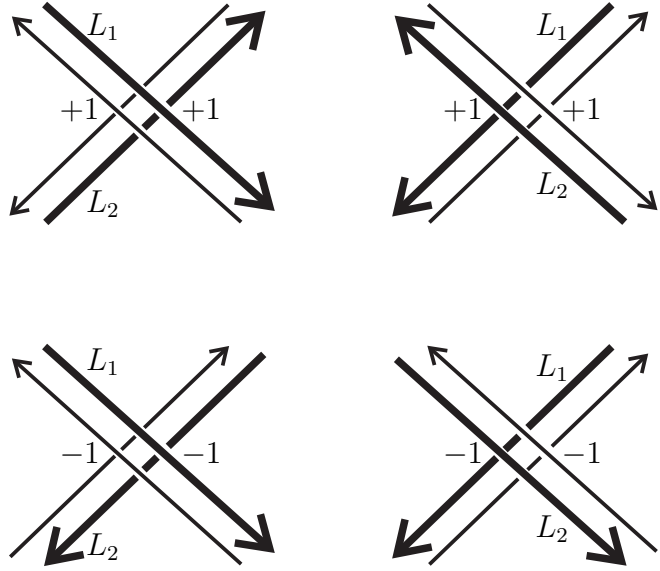}

\put(30, 135){$L_{2}$}
\put(30, 202){$L_{1}$}
\put(200, 202){$L_{1}$}
\put(200, 142){$L_{2}$}
\put(30, 75){$L_{1}$}
\put(30, 5){$L_{2}$}
\put(200, 73){$L_{1}$}
\put(200, 12){$L_{2}$}
\put(66, 40){$-1$}
\put(20, 40){$-1$}
\put(66, 170){$+1$}
\put(20, 170){$+1$}
\put(165, 170){$+1$}
\put(210, 170){$+1$}
\put(165, 40){$-1$}
\put(210, 40){$-1$}

\end{overpic}
\caption{Four possible configurations of $L'$ near a crossing of the outside parts of $L_{1}$ and $L_{2}$. The thin arcs are parts of $L'$.}
\label{fig: writhe}
\end{figure}

We compute the framings of $L_1$, $L_2$ and $L'$ induced by $S$.

\begin{lemma} \label{SurfaceFraming} (1) For $i=1,2$, the framing of $L_{i}$ induced by $S$ is $tb(L_i)+1$ with respect to the Seifert surface framing of $L_i$. \\ (2) The framing of $L'$  induced by $S$  is $tb(L_1)+tb(L_2)+2(l+1)$ with respect to the Seifert surface framing of $L'$; that is, the framing of $L'$ induced by $S$ coincides with the contact framing of $L'$.
\end{lemma}

\begin{proof}
(1) For $i=1,2$, the framing of $L_{i}$ induced by $S$, with respect to the Seifert surface framing of $L_i$, is the linking number of $L_{i}$ and its push-off in the interior of $S$.  Note that the push-off of the outside part of $L_i$ in the interior of $S$ is isotopic to a Legendrian push-off of the outside part of $L_{i}$.  So it is easy to know that the framing of $L_{i}$ induced by $S$ is $tb(L_i)+1$ with respect to the Seifert surface framing of $L_i$.

(2) The framing of $L'$ induced by $S$, with respect to the Seifert surface framing of $L'$, is the linking number $lk(L',L'_0)$, where $L_0'$ is a push-off of $L'$ in the interior of $S$. We compute $lk(L',L_0')$ as the number of crossings where $L_0'$ crosses under $L'$, counted with sign. A similar argument as in the proof of Lemma~\ref{ContactFraming}
shows that the outside parts of $L'$ and $L_0'$ contribute $tb(L_1)+tb(L_2)+2(l+1)$ to $lk(L',L_0')$. It is easy to see that the inside parts of $L'$ and $L_0'$ contribute $0$
to $lk(L',L_0')$. Hence the framing of $L'$ induced by $S$ is $tb(L_1)+tb(L_2)+2(l+1)$ with respect to the Seifert surface framing of $L'$. By Lemma~\ref{ContactFraming}, this
framing coincides with the contact framing of $L'$.
\end{proof}

By Lemma~\ref{SurfaceFraming} (1), after we perform contact $(+1)$-surgery along the Legendrian link $L$, $S$ caps off to a disk $D$ with boundary $L'$. According to Lemma~\ref{SurfaceFraming} (2), the contact framing of $L'$ equals the framing of $L'$ induced by the disk $D$. Hence $D$ is an overtwisted disk and the contact 3-manifold after contact $(+1)$-surgery is overtwisted.
\end{proof}

\begin{example} \label{Example4}
Contact $(+1)$-surgery along the Legendrian link in Figure~\ref{fig: l1} yields an overtwisted contact 3-manifold. This is because the dashed box in Figure~\ref{fig: lk5} contains the configuration in Figure~\ref{fig: ot3}(c).
\end{example}

\begin{figure}[htb]
\begin{overpic}
{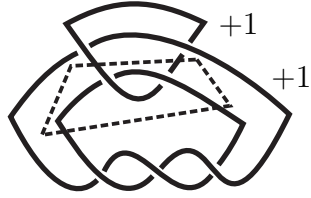}
\put(80, 60){$+1$}
\put(100, 40){$+1$}

\end{overpic}
\caption{A configuration in the dashed box. }
\label{fig: lk5}
\end{figure}

We can  transform  the four configurations in Figure~\ref{fig: ot3} to that in Figure~\ref{fig: ot4} through Legendrian Reidemeister moves. So we have the following corollary.

\begin{corollary}\label{overtwisted1}
Suppose there exists a front projection of  a Legendrian two-component link $L=L_1\cup L_2$ in the standard contact 3-sphere $(S^3, \xi_{std})$ that contains one of the configurations exhibited in Figure~\ref{fig: ot4}, then contact $(+1)$-surgery on $(S^3, \xi_{std})$ along $L$ yields an overtwisted contact 3-manifold.
\end{corollary}

\begin{figure}[htb]
\begin{overpic}
{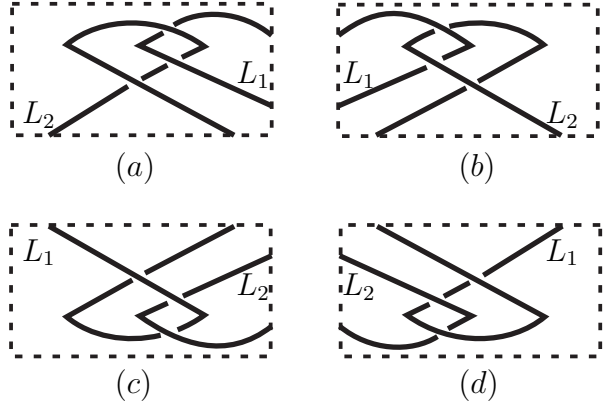}

\put(86, 105){$L_{1}$}
\put(5, 90){$L_{2}$}
\put(40, 70){$(a)$}

\put(126, 105){$L_{1}$}
\put(203, 90){$L_{2}$}
\put(170, 70){$(b)$}

\put(86, 25){$L_{2}$}
\put(5, 38){$L_{1}$}
\put(40, -12){$(c)$}

\put(126, 25){$L_{2}$}
\put(203, 38){$L_{1}$}
\put(170, -12){$(d)$}

\end{overpic}
\vspace{3mm}
\caption{Four configurations in a front projection of a Legendrian two-component link $L$. }
\label{fig: ot4}
\end{figure}


\begin{proof}

We can transform the configuration in Figure~\ref{fig: ot3}(a) to the configuration in Figure~\ref{fig: ot4}(a) through Legendrian Reidemeister moves and an isotopy demonstrated in Figure~\ref{fig: ot5}. The other cases are similar.
 \begin{figure}[htb]
\begin{overpic}
{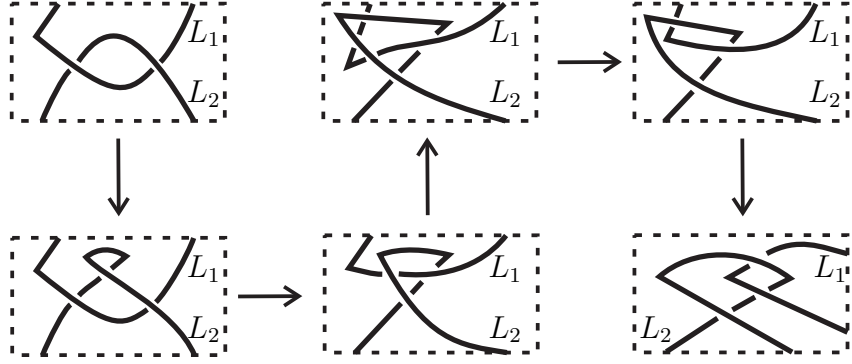}

\put(68, 120){$L_{1}$}
\put(68, 97){$L_{2}$}

\put(182, 120){$L_{1}$}
\put(182, 97){$L_{2}$}

\put(303, 120){$L_{1}$}
\put(303, 97){$L_{2}$}

\put(68, 31){$L_{1}$}
\put(68, 8){$L_{2}$}

\put(182, 31){$L_{1}$}
\put(182, 8){$L_{2}$}

\put(305, 31){$L_{1}$}
\put(239, 8){$L_{2}$}

\end{overpic}
\caption{A Legendrian isotopy. The first four arrows are Legendrian Reidemeister moves. The last arrow is an isotopy.}
\label{fig: ot5}
\end{figure}
\end{proof}

\begin{example} \label{Example5}
In Figure~\ref{fig: ot6}, $L_1'$ and $L_2'$ are parts of $L_1$ and $L_2$, respectively. We claim that contact $(+1)$-surgery along the Legendrian link $L_1\cup L_2$ in the left of Figure~\ref{fig: ot6} yields an overtwisted contact 3-manifold. This is because we can transform the Legendrian link in the left of Figure~\ref{fig: ot6} to that in the right of  Figure~\ref{fig: ot6} which contains the configuration in Figure~\ref{fig: ot4}(a) through Legendrian Reidemeister moves.

Consequently, contact $(+1)$-surgery along the Legendrian link in Figure~\ref{fig: l2} yields an overtwisted contact 3-manifold.
\end{example}

\begin{figure}[htb]
\begin{overpic}
{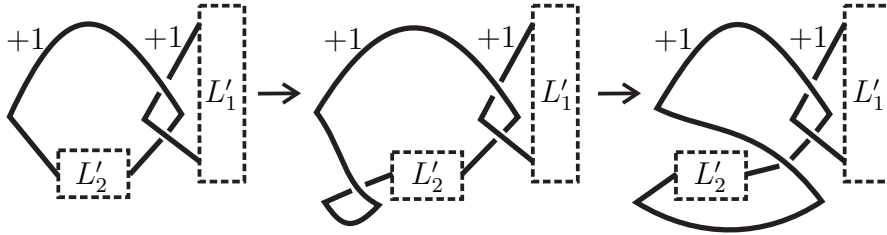}

\put(75, 50){$L_{1}'$}
\put(26, 19){$L_{2}'$}

\put(202, 50){$L_{1}'$}
\put(260, 19){$L_{2}'$}

\put(320, 50){$L_{1}'$}
\put(153, 19){$L_{2}'$}

\put(0, 70){$+1$}
\put(52, 71){$+1$}

\put(120, 70){$+1$}
\put(178, 71){$+1$}

\put(245, 70){$+1$}
\put(296, 71){$+1$}

\end{overpic}
\caption{An example of a contact $(+1)$-surgery yielding an overtwisted contact 3-manifold. The arrows are Legendrian Reidemeister moves. }
\label{fig: ot6}
\end{figure}

We conclude this section with an example whose tightness is still unclear.  It is interesting in the sense that they provide potential candidates for tight contact 3-manifolds with vanishing contact invariant obtained from $(+1)$-surgery along Legendrian links.

\begin{example}
In Figure~\ref{fig: xp3}, we consider contact $(+1)$-surgeries along the following Legendrian links: The first link consists of a Legendrian $\overline{9_{46}}$ with Thurston-Bennequin invariant $-1$ and rotation number $0$, and its Legendrian pushoff.  The second link is constructed by performing a Legendrian connected sum of the upper component of the first link with the Legendrian right handed trefoil with Thurston-Bennequin invariant $1$.

\begin{figure}[htb]
\begin{overpic}
{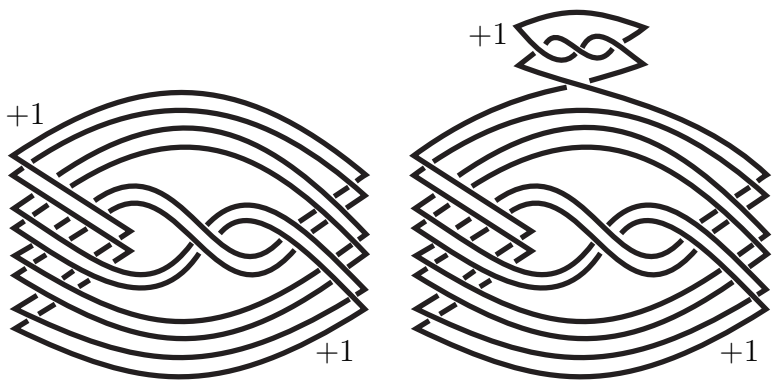}
\put(175, 127){$+1$}
\put(270, 7){$+1$}
\put(0, 97){$+1$}
\put(117, 7){$+1$}

\end{overpic}
\caption{The first link consists of a Legendrian $\overline{9_{46}}$ with Thurston-Bennequin invariant $-1$ and rotation number $0$ and its Legendrian pushoff; the second link is obtained from the first link by a Legendrian connected sum with a Legendrian right handed trefoil. }
\label{fig: xp3}
\end{figure}

Although contact $(+1)$-surgery along each knot component of the depicted links has nonvanishing contact invariant, contact $(+1)$-surgeries along both links result in contact 3-manifolds with vanishing contact invariants.  This follows readily from Theorem \ref{thm: main}.  In fact, contact $(+1)$-surgery along the first link is contactomorphic to contact $\frac{1}{2}$-surgery along $\overline{9_{46}}$, which is also known to have vanishing contact invariant from Mark-Tosun \cite[Theorem 1.2]{mt}.

On the other hand, we have not been able to determine whether the above contact 3-manifolds are overtwisted or not using the techniques in this section.

\end{example}

\end{document}